\documentclass[12pt,english,final]{amsart}
\usepackage{amsmath,amsthm,amssymb,amsfonts,amscd, array,latexsym}
\usepackage[noadjust]{cite}
\usepackage[T2A]{fontenc}
\usepackage[cp1251]{inputenc}
\usepackage[english]{babel}
\usepackage{diagmac2} %% for diagrams
\usepackage{tikz}
\usetikzlibrary{positioning,calc,arrows,arrows.meta}

\usepackage{color} %for using colors

\usepackage[notref,notcite]{showkeys} % потом закомментировать

\newtheorem{theorem}{Theorem}[section]
\newtheorem{corollary}[theorem]{Corollary}
\newtheorem{lemma}[theorem]{Lemma}
\newtheorem{proposition}[theorem]{Proposition}

\theoremstyle{definition}
\newtheorem{definition}[theorem]{Definition}
\newtheorem{example}[theorem]{Example}

\theoremstyle{remark}
\newtheorem{remark}[theorem]{Remark}
\numberwithin{equation}{section}

\makeatletter

\renewcommand{\p@enumii}{}

\makeatother

\newcommand{\RR}{\mathbb{R}}

\newcommand{\Sym}{\operatorname{\mathbf{Sym}}}
\newcommand{\Cs}{\operatorname{\mathbf{Cs}}}
\newcommand{\Iso}{\operatorname{\mathbf{Iso}}}

\def\<#1>{\langle #1 \rangle}

\newbox\onebox
\newcommand{\coherent}[1]{\mathbin{\setbox\onebox=\hbox{$=$}\lower0.7\ht%
\onebox\hbox{$\stackrel{#1}{=}$}}}

\newcommand{\acr}{\newline\indent}

\begin{document}

\title{When all Permutations are Combinatorial Similarities}

\author{Viktoriia Bilet}
\address{\textbf{Viktoriia Bilet}\acr
Department of Theory of Functions \acr
Institute of Applied Mathematics and Mechanics of NASU\acr
Dobrovolskogo str. 1, Slovyansk 84100, Ukraine}
\email{viktoriiabilet@gmail.com}

\author{Oleksiy Dovgoshey}
\address{\textbf{Oleksiy Dovgoshey}\acr
Department of Theory of Functions \acr
Institute of Applied Mathematics and Mechanics of NASU \acr
Dobrovolskogo str. 1, Slovyansk 84100, Ukraine \acr
Institut fuer Mathematik Universitaet zu Luebeck\acr
Ratzeburger Allee 160, D-23562 Luebeck, Deutschland}

\email{oleksiy.dovgoshey@gmail.com}


\begin{abstract}
Let \((X, d)\) be a semimetric space. A permutation \(\Phi\) of the set \(X\) is a combinatorial self similarity of \((X, d)\) if there is a bijective function \(f \colon d(X^2) \to d(X^2)\) such that
\[
d(x, y) = f(d(\Phi(x), \Phi(y)))
\]
for all \(x\), \(y \in X\). We describe the set of all semimetrics \(\rho\) on an arbitrary nonempty set \(Y\) for which every permutation of \(Y\) is a combinatorial self similarity of \((Y, \rho)\).
\end{abstract}

% -----------------------------------------------------------
\maketitle
% -----------------------------------------------------------

\section{Introduction}

Let us start from the classical notion of metric space introduced by Maurice Fr\'{e}chet in his thesis \cite{Fre1906RdCMdP}.

A \textit{metric} on a set \(X\) is a function \(d\colon X^{2} \to \RR\) such that for all \(x\), \(y\), \(z \in X\):
\begin{enumerate}
\item \(d(x,y) \geqslant 0\) with equality if and only if \(x=y\), the \emph{positivity property};
\item \(d(x,y)=d(y,x)\), the \emph{symmetry property};
\item \(d(x, y)\leq d(x, z) + d(z, y)\), the \emph{triangle inequality}.
\end{enumerate}
Here and what follows \(X^2\) is the Cartesian square of the set \(X\), i.e., \(X^2\) is the set of all ordered pairs \(\<x, y>\), where \(x\), \(y \in X\).

An useful generalization of the notion of metric space is the concept of semimetric space.

\begin{definition}\label{ch2:d2}
Let \(X\) be a set and let \(d \colon X^{2} \to \RR\) be a symmetric function. The function \(d\) is a \emph{semimetric} on \(X\) if it satisfies the positivity property.
\end{definition}

If \(d\) is a semimetric on \(X\), we say that \((X, d)\) is a \emph{semimetric space}. The semimetric spaces also were first considered by Fr\'{e}chet in \cite{Fre1906RdCMdP}, where he called them ``classes \(E\)''.  It should be noted that a different terminology is used to denote the class of semimetric spaces: distances spaces \cite{Gre2016CMP}, spaces endowed with dissimilarities \cite{DF1998DM}, symmetric spaces for the case of topological spaces with the topology generated by \(d\) \cite[Capter~10]{KV1984}. We inherit the terminology from Wilson's paper \cite{Wilson1931}, Blumenthal's book \cite{Blumenthal1953}, and many recent papers \cite{BP2017JNCA, DP2013AMH, DH2017JFPTA, JT2020R, KS2015JFPTA, PS2022AFM}.

In the next definition we will denote by \(d(X^2)\) the range of \(d\),
\[
d(X^2) := \{d(x, y) \colon x, y \in X\}.
\]

\begin{definition}[\cite{DLAMH2020}]\label{d1.2}
Let \((X, d)\) and \((Y, \rho)\) be semimetric spaces. The spaces \((X, d)\) and \((Y, \rho)\) are \emph{combinatorially similar} if there exist bijections \(\Psi \colon Y \to X\) and \(f \colon d(X^{2}) \to \rho(Y^{2})\) such that
\begin{equation}\label{d1.2:e1}
\rho(x,y) = f\bigl(d(\Psi(x), \Psi(y))\bigr)
\end{equation}
for all \(x\), \(y \in Y\). In this case, we will say that \(\Psi \colon Y \to X\) is a \emph{combinatorial similarity}.
\end{definition}

\begin{remark}\label{r1.3}
If \(Y \xrightarrow{\Psi} X\) is a combinatorial similarity of spaces \((X, d)\) and \((Y, \rho)\), then the inverse mapping \(X \xrightarrow{\Psi^{-1}} Y\) is also a combinatorial similarity of these spaces. Moreover, if \((Z, \delta)\) is a semimetric space and \(X \xrightarrow{\Phi} Z\) is a combinatorial similarity of \((Z, \delta)\) and \((X, d)\), then the composition \(Y \xrightarrow{\Psi} X \xrightarrow{\Phi} Z\) is also a combinatorial similarity. In particular, for every semimetric space \((X, d)\), the set of all combinatorial self similarities \(X \to X\) is a subgroup of the symmetric group \(\Sym(X)\) of all permutations on \(X\). In what follows this subgroup will be denoted as \(\Cs(X, d)\).
\end{remark}

Let us consider some examples of combinatorial similarities.

\begin{example}\label{ex1.3}
Semimetric spaces \((X, d)\) and \((Y, \rho)\) are \emph{isometric} if there is a bijection \(\Psi \colon Y \to X\) such that
\[
\rho(x, y) = d(\Psi(x), \Psi(y))
\]
for all \(x\), \(y \in Y\). In this case we will say that \(\Psi\) is an \emph{isometry} of \((X, d)\) and \((Y, \rho)\). It is easy to see that all isometrics are combinatorial similarities. Moreover, a combinatorial similarity \(\Psi \colon Y \to X\) is an isometry if and only if \eqref{d1.2:e1} holds for all \(x\), \(y \in Y\) with \(f(t) = t\) for every \(t \in d(X^2)\).
\end{example}

The group of all self isometries of a semimetric space \((X, d)\) will be denoted as \(\Iso(X, d)\).

\begin{example}\label{ex5.32}
Let \(\Phi \colon X \to Y\) be a bijection and let \(d \colon X^2 \to \RR\) and \(\rho \colon Y^2 \to \RR\) be some semimetrics. The mapping \(\Phi\) is a \emph{weak similarity} of \((X, d)\) and \((Y, \rho)\) if and only if the equivalence
\begin{equation}\label{ex5.32:e1}
(d(x, y) \leqslant d(w, z)) \Leftrightarrow (\rho(\Phi(x), \Phi(y)) \leqslant \rho(\Phi(w), \Phi(z)))
\end{equation}
is valid for all \(x\), \(y\), \(z\), \(w \in X\).
\end{example}

Equivalence~\eqref{ex5.32:e1} evidently implies the validity of
\begin{equation*}
(d(x, y) = d(w, z)) \Leftrightarrow (\rho(\Phi(x), \Phi(y)) = \rho(\Phi(w), \Phi(z))).
\end{equation*}
Thus, every weak similarity is a combinatorial similarity (see Lemma~\ref{l2.3} below). It is interesting to note that, the inverse statement is also valid forth case when  \((X, d)\) and \((Y, \rho)\) are ultrametric spaces. Every combinatorial similarity of \((X, d)\) and \((Y, \rho)\) is a weak similarity of these spaces (see Theorem 4.7 in \cite{DovBBMSSS2020}).  Some questions connected with the weak similarities and combinatorial similarities were also studied in \cite{DLAMH2020, Dov2019IEJA, BDS2021pNUAA}. The weak similarities of finite ultrametric and semimetric spaces were considered by E.~Petrov in \cite{Pet2018pNUAA}.

The combinatorial similarities are the main morphisms of semimetric spaces which will be studied in the paper. Let us introduce now the subclasses of semimetric spaces which will be important for us.

Let \((X, d)\) be a metric space. The metric \(d\) is said to be \emph{strongly rigid} if, for all \(x\), \(y\), \(u\), \(v \in X\), the condition
\begin{equation}\label{e1.3}
d(x, y) = d(u, v) \neq 0
\end{equation}
implies
\begin{equation}\label{e3.6}
(x = u \text{ and } y = v) \text{ or } (x = v \text{ and } y = u).
\end{equation}
(Some properties of strongly rigid metric spaces are described in \cite{Janos1972, Martin1977, BDKP2017AASFM, DLAMH2020, DS2021aa}.)

The concept of strongly rigid metric can be naturally generalized to the concept of \emph{strongly rigid semimetric}.

\begin{definition}\label{d4.3}
Let \((X, d)\) be a semimetric space. The semimetric \(d\) is \emph{strongly rigid} if \eqref{e1.3} implies \eqref{e3.6} for all \(x\), \(y\), \(u\), \(v \in X\).
\end{definition}

It is easy to prove that the equality \(\Cs(X, d) = \Sym(X)\) holds for every strongly rigid \((X, d)\) (see proof of Lemma~\ref{l4.2}).

The following generalization of strongly rigid semimetric spaces was recently introduced in \cite{DS2021aa}.

\begin{definition}\label{d1.6}
A semimetric space \((X, d)\) is \emph{weakly rigid} if every three-point subspace of \((X, d)\) is strongly rigid.
\end{definition}

We say that a semimetric \(d \colon X^{2} \to \RR\) is \emph{discrete} if the inequality
\begin{equation}\label{e3.5}
|d(X^{2})| \leqslant 2
\end{equation}
holds, where \(|d(X^{2})|\) is the cardinal number of the set \(d(X^{2})\). Thus, a semimetric \(d \colon X^{2} \to \RR\) is discrete if and only if there is \(k >0\) such that the equality
\begin{equation}\label{e3.5*}
d(x, y) = k
\end{equation}
holds for all different \(x\), \(y \in X\). It follows directly from the definitions that a semimetric space \((X, d)\) is discrete if and only if \(\Iso(X, d) = \Sym(X)\).

The goal of the paper is to describe all possible semimetric spaces \((X, d)\) for which the equality
\begin{equation}\label{e1.6}
\Cs(X, d) = \Sym(X)
\end{equation}
holds.

The paper is organized as follows.

Proposition~\ref{p4.4} gives us a characterization of discrete \(d \colon X^2 \to \RR\) in terms of combinatorial self similarities of \((X, d)\).

The main result of the paper, Theorem~\ref{t2.4}, completely characterizes the structure of all semimetric spaces \((X, d)\) satisfying equality \eqref{e1.6}.

In Corollary~\ref{c2.10} we show that equality \eqref{e1.6} is equivalent to the equality
\[
\Iso(X, d) = \Sym(X)
\]
if \(X\) is big enough.

\section{Permutations and combinatorial self similarities}

Let us start from an example of a four-point metric space \((Z, d)\) which is neither strongly rigid nor discrete but satisfies the equality \(\Cs(Z, d) = \Sym(Z)\).

\begin{example}\label{ex2.1}
Let \(Z = \{z_1, z_2, z_3, z_4\}\) be the four-point subset of the complex plane \(\mathbb{C}\),
\[
z_1 = 0+0i, \quad z_2 = 0+3i, \quad z_3 = 4+3i, \quad z_4 = 4+0i,
\]
and let \(d\) be the restriction of the usual Euclidean metric on \(Z^2\). The equalities \(d(z_1, z_2) = d(z_3, z_4) = 3\) and \(d(z_1, z_4) = d(z_2, z_3) = 4\) imply that \((Z, d)\) is neither strongly rigid nor discrete but it can be proved directly that \(\Cs(Z, d) = \Sym(Z)\) holds.
\end{example}

\begin{lemma}\label{l2.3}
Let \((X, d)\) and \((Y, \rho)\) be semimetric spaces. A bijection \(\Phi \colon Y \to X\) is a combinatorial similarity iff the equivalence
\[
\bigl(\rho(x, y) = \rho(u, v)\bigr) \Leftrightarrow \Bigl(d\bigl(\Phi(x), \Phi(y)\bigr) = d\bigl(\Phi(u), \Phi(v)\bigr)\Bigr)
\]
is valid for all \(x\), \(y\), \(u\), \(v \in Y\).
\end{lemma}

\begin{proof}
It follows from Definition~\ref{d1.2}.
\end{proof}

The next proposition characterizes all semimetric spaces which are combinatorially similar to the rectangle from Example~\ref{ex2.1}.

\begin{proposition}\label{p2.2}
Let \(\rho \colon X^2 \to \RR\) be a semimetric on the set \(X\) with \(|X| \geqslant 4\). Then the following conditions are equivalent:
\begin{enumerate}
\item \label{p2.2:c1} \((X, \rho)\) is combinatorially similar to the metric space \((Z, d)\) from Example~\ref{ex2.1}.
\item \label{p2.2:c2} \((X, \rho)\) is weakly rigid and, moreover, all three-point subspaces of \((X, \rho)\) are isomeric.
\end{enumerate}
\end{proposition}

\begin{proof}
\(\ref{p2.2:c1} \Rightarrow \ref{p2.2:c2}\). The validity of this implication can be proved using Lemma~\ref{l2.3}.

\(\ref{p2.2:c2} \Rightarrow \ref{p2.2:c1}\). Let \ref{p2.2:c2} hold. Then there are different numbers \(a\), \(b\), \(c \in (0, \infty)\) such that, for every three-point subspace \(T\) of \((X, \rho)\), we have
\begin{equation}\label{p2.2:e1}
\rho(X^2) = \rho(T^2) = \{0, a, b, c\}.
\end{equation}

Using \eqref{p2.2:e1} it is easy to prove that \(|X| = 4\). Indeed, let \(p\) be a point of \(X\). If \(|X| \geqslant 5\) holds, then we have the inequality
\[
|X \setminus \{p\}| \geqslant 4.
\]
Consequently, by Pigeonhole principle, there are some different points \(x_1\), \(x_2 \in X \setminus \{p\}\) such that
\[
\rho(x_1, p) = \rho(x_2, p).
\]
Thus, \((X, \rho)\) is not weakly rigid contrary to \ref{p2.2:c2}. It implies the inequality
\begin{equation}\label{p2.2:e1.1}
|X| \leqslant 4.
\end{equation}
By condition, we have \(|X| \geqslant 4\). The last inequality and \eqref{p2.2:e1.1} give us \(|X| = 4\).

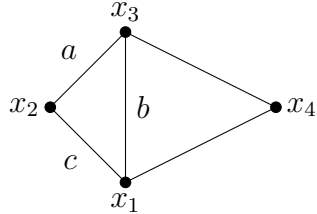
\begin{figure}[htb]
\begin{tikzpicture}
\coordinate [label=below: \(x_1\)] (x1) at (1, -1);
\coordinate [label=left: \(x_2\)] (x2) at (0, 0);
\coordinate [label=above: \(x_3\)] (x3) at (1, 1);
\coordinate [label=right: \(x_4\)] (x4) at (3, 0);

\draw [fill] (x1) circle (2pt);
\draw [fill] (x2) circle (2pt);
\draw [fill] (x3) circle (2pt);
\draw [fill] (x4) circle (2pt);

\draw (x1) -- node [below left] {\(c\)} (x2) -- node [above left] {\(a\)} (x3) -- (x4) -- (x1);
\draw (x1) -- node [right] {\(b\)} (x3);
\end{tikzpicture}
\caption{Add the point \(x_4\) to the triangle \(\{x_1, x_2, x_3\}\).}
\label{fig3}
\end{figure}

Let \(x_1\), \(x_2\), \(x_3\) be points of \(X\) such that
\[
\rho(x_1, x_2) = c, \quad \rho(x_2, x_3) = a \quad \text{and}\quad \rho(x_3, x_1) = b.
\]
Since \(|X| = 4\) holds, the set \(X \setminus \{x_1, x_2, x_3\}\) contains a unique point \(x_4\). Let us consider the triangle \(\{x_1, x_3, x_4\}\). From \eqref{p2.2:e1} with \(T = \{x_2, x_3, x_4\}\) and the equality \(\rho(x_3, x_1) = b\) it follows that we have either
\begin{equation}\label{p2.2:e2}
\rho(x_3, x_4) = a
\end{equation}
or \(\rho(x_3, x_4) = c\). If \eqref{p2.2:e2} holds, then we obtain \(\rho(x_2, x_3) = \rho(x_3, x_4)\). Hence, \((X, \rho)\) is not weakly rigid contrary to \ref{p2.2:c2}. Thus, we have \(\rho(x_3, x_4) = c\). The last equality, the equality \(\rho(x_1, x_3) = b\), and \eqref{p2.2:e1} imply the equality \(\rho(x_4, x_1) = a\).

Let us consider now the triangle \(\{x_2, x_3, x_4\}\). Then using \eqref{p2.2:e1} with \(T = \{x_2, x_3, x_4\}\) we can prove that \(\rho(x_3, x_4) = c\) and \(\rho(x_2, x_3) = a\) imply \(\rho(x_3, x_1) = b\). To complete the proof it is sufficient to note that the bijection \(F \colon Z \to X\), \(F(z_j) = x_j\), \(j = 1\), \(\ldots\), \(4\), is a combinatorial similarity of the semimetric space \((X, \rho)\) and the rectangle \((Z, d)\) (see Figure~\ref{fig4} below). \qedhere

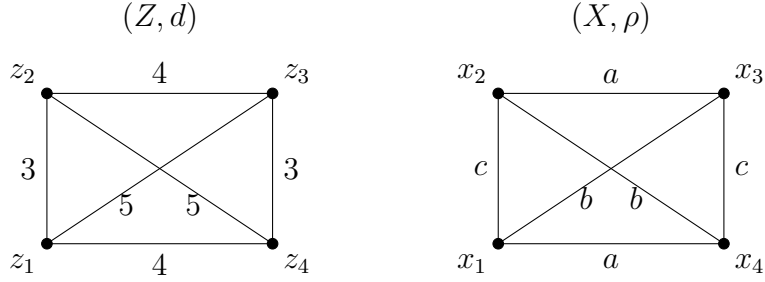
\begin{figure}[hb]
\begin{tikzpicture}
\def\dx{3cm}
\def\dy{2cm}
\coordinate [label=below left: \(z_1\)] (x1) at (0, 0);
\coordinate [label=above left: \(z_2\)] (x2) at (0, \dy);
\coordinate [label=above right: \(z_3\)] (x3) at (\dx, \dy);
\coordinate [label=below right: \(z_4\)] (x4) at (\dx, 0);

\draw [fill] (x1) circle (2pt);
\draw [fill] (x2) circle (2pt);
\draw [fill] (x3) circle (2pt);
\draw [fill] (x4) circle (2pt);

\node at (\dx/2, 1.5*\dy) {\((Z, d)\)};

\draw (x1) -- node [left] {\(3\)} (x2) -- node [above] {\(4\)} (x3) --node [right] {\(3\)}  (x4) -- node [below] {\(4\)} (x1);
\draw (x1) -- node [below right=8pt] {\(5\)} (x3);
\draw (x2) -- node [below left=8pt] {\(5\)} (x4);

\begin{scope}[xshift=6cm]
\coordinate [label=below left: \(x_1\)] (x1) at (0, 0);
\coordinate [label=above left: \(x_2\)] (x2) at (0, \dy);
\coordinate [label=above right: \(x_3\)] (x3) at (\dx, \dy);
\coordinate [label=below right: \(x_4\)] (x4) at (\dx, 0);

\draw [fill] (x1) circle (2pt);
\draw [fill] (x2) circle (2pt);
\draw [fill] (x3) circle (2pt);
\draw [fill] (x4) circle (2pt);

\node at (\dx/2, 1.5*\dy) {\((X, \rho)\)};

\draw (x1) -- node [left] {\(c\)} (x2) -- node [above] {\(a\)} (x3) --node [right] {\(c\)}  (x4) -- node [below] {\(a\)} (x1);
\draw (x1) -- node [below right=4pt] {\(b\)} (x3);
\draw (x2) -- node [below left=4pt] {\(b\)} (x4);
\end{scope}
\end{tikzpicture}
\caption{The space \((X, \rho)\) and \((Z, d)\) are combinatorially similar.}
\label{fig4}
\end{figure}
\end{proof}

\begin{lemma}\label{l4.2}
Let \((X, d)\) be a three-point metric space. Then the following statements are equivalent:
\begin{enumerate}
\item\label{l4.2:s1} Either \((X, d)\) is strongly rigid or \((X, d)\) is discrete.
\item\label{l4.2:s2} \(\Cs(X, d) = \Sym(X)\) holds.
\end{enumerate}
\end{lemma}

\begin{proof}
\(\ref{l4.2:s1} \Rightarrow \ref{l4.2:s2}\). Let \ref{t2.4:s1} hold. We must prove that every permutation of \(X\) is a combinatorial self similarity of \((X, d)\). Let \(\Phi\) belong to \(\Sym(X)\). By Lemma~\ref{l2.3}, the mapping \(\Phi\) is a combinatorial similarity iff
\begin{equation}\label{l4.2:e1}
\bigl(d(x, y) = d(u, v)\bigr) \Leftrightarrow \bigl(d(\Phi(x), \Phi(y)) = d(\Phi(u), \Phi(v))\bigr)
\end{equation}
holds for all \(x\), \(y\), \(u\), \(v \in X\). It is clear that \eqref{l4.2:e1} is valid if \(x = y\) or \(u = v\). Thus, it suffices to prove the validity of \eqref{l4.2:e1} for the case when
\begin{equation}\label{l4.2:e2}
d(x, y) = d(u, v) \neq 0.
\end{equation}

If \(d\) is strongly rigid, then \eqref{l4.2:e2} implies
\begin{equation}\label{l4.2:e3}
\<x, y> = \<u, v> \quad \text{or} \quad \<x, y> = \<v, u>
\end{equation}
(see \eqref{e3.6}). Since \(\Phi \colon X \to X\) is bijective, \eqref{l4.2:e3} can be written in the following equivalent form
\[
\<\Phi(x), \Phi(y)> = \<\Phi(u), \Phi(v)> \quad \text{or} \quad \<\Phi(x), \Phi(y)> = \<\Phi(v), \Phi(u)>.
\]
The right equality from \eqref{l4.2:e1} follows.

If \(d\) is discrete, then we have \(\Sym(X) = \Iso(X, d)\). That implies \ref{l4.2:s2} because
\[
\Iso(X, d) \subseteq \Cs(X, d) \subseteq \Sym(X).
\]

\(\ref{l4.2:s2} \Rightarrow \ref{l4.2:s1}\). Let \ref{l4.2:s2} hold. If \((X, d)\) is neither rigid nor discrete, then, we can numbered the points of \(X\) and find \(a\), \(b \in (0, \infty)\) such that \(X = \{x_1, x_2, x_3\}\) and
\begin{equation}\label{l4.2:e4}
d(x_1, x_2) = d(x_2, x_3) = a \neq b = d(x_3, x_1).
\end{equation}
Let consider the permutation
\begin{equation}\label{l4.2:e5}
\Phi = \begin{pmatrix}
x_1 & x_2 & x_3 \\
x_2 & x_3 & x_1
\end{pmatrix}.
\end{equation}
Using \eqref{l4.2:e4} and \eqref{l4.2:e5} we obtain
\[
d(\Phi(x_1), \Phi(x_2)) = d(x_2, x_3) = a
\]
and
\[
d(\Phi(x_2), \Phi(x_3)) = d(x_3, x_1) = b.
\]
Hence, the permutation \(\Phi\) is not a combinatorial self similarity by Lemma~\ref{l2.3}. The last statement contradicts \ref{l4.2:s2}.
\end{proof}

\begin{lemma}\label{l4.3}
Let \((X, d)\) be a semimetric space and let \(Y\) be a nonempty subset of \(X\). If \(\Cs(X, d) = \Sym(X)\) holds, then we have the equality \(\Cs(Y, d|_{Y^2}) = \Sym(Y)\).
\end{lemma}

\begin{proof}
Each permutation of \(Y\) can extended to a permutation of \(X\) and, in addition, if \(\Psi \colon X \to X\) is a combinatorial self similarity, then \(\Phi = \Psi|_{Y}\) is also a combinatorial self similarity by Lemma~\ref{l2.3}.
\end{proof}

\begin{lemma}\label{l2.6}
Let \((X, d)\) and \((Y, \rho)\) be combinatorially similar. Then the equality \(\Cs(X, d) = \Sym(X)\) holds if and only if we have \(\Cs(Y, \rho) = \Sym(Y)\).
\end{lemma}

\begin{proof}
Let \(\Phi \colon X \to Y\) be a combinatorial similarity and let
\begin{equation}\label{l2.6:e1}
\Cs(Y, \rho) = \Sym(Y)
\end{equation}
hold. Let us consider an arbitrary \(F \in \Sym(X)\). Then the mapping
\[
Y \xrightarrow{\Phi^{-1}} X \xrightarrow{F} X \xrightarrow{\Phi} Y
\]
is a permutation of \(Y\). Let us denote this permutation by \(\Psi\). Then \(\Psi\) belongs to \(\Cs(Y, \rho)\) by \eqref{l2.6:e1}. Since the permutation \(F\) of \(X\) coincides with the mapping
\[
X \xrightarrow{\Phi} Y \xrightarrow{\Phi^{-1}} X \xrightarrow{F} X \xrightarrow{\Phi} Y \xrightarrow{\Phi^{-1}} X,
\]
\(F\) belongs to \(\Cs(X, d)\) by Remark~\ref{r1.3}. Thus, \eqref{l2.6:e1} implies the equality
\begin{equation}\label{l2.6:e2}
\Cs(X, d) = \Sym(X).
\end{equation}
Arguing in a similar way, we can prove that \eqref{l2.6:e1} follows from \eqref{l2.6:e2}.
\end{proof}

\begin{proposition}\label{p4.4}
Let \((X, d)\) be a semimetric space with \(|X| \geqslant 3\). Then the following statements are equivalent:
\begin{enumerate}
\item\label{p4.4:s1} \((X, d)\) is discrete.
\item\label{p4.4:s2} \((X, d)\) contains an equilateral triangle and the equality \(\Cs(X, d) = \Sym(X)\) holds.
\end{enumerate}
\end{proposition}

\begin{proof}
\(\ref{p4.4:s1} \Rightarrow \ref{p4.4:s2}\). Let \((X, d)\) be discrete. Then \((X, d)\) contains an equilateral triangle because \(|X| \geqslant 3\). Moreover, the equality \(\Iso(X, d) = \Sym(X)\) and the inclusions
\[
\Iso(X, d) \subseteq \Cs(X, d) \subseteq \Sym(X)
\]
imply \(\Cs(X, d) = \Sym(X)\).

\(\ref{p4.4:s2} \Rightarrow \ref{p4.4:s1}\). Let \ref{p4.4:s2} hold. Then there are some points \(x_1\), \(x_2\), \(x_3 \in X\) and \(a > 0\) such that
\begin{equation}\label{p4.4:e1}
d(x_1, x_2) = d(x_2, x_3) = d(x_3, x_1) = a.
\end{equation}
Let \(p\) and \(q\) be distinct points of \((X, d)\). We must prove the equality
\begin{equation}\label{p4.4:e2}
d(p, q) = a.
\end{equation}
The last equality trivially holds if \(\{p, q\} \subseteq \{x_1, x_2, x_3\}\). Let us consider the case when \(p \in \{x_1, x_2, x_3\}\), but \(q \notin \{x_1, x_2, x_3\}\). Without loss of generality we may assume that \(p = x_1\). Let us define a permutation \(F \colon X \to X\) as
\begin{equation}\label{p4.4:e3}
F(x) := \begin{cases}
q & \text{if } x = x_2\\
x_2 & \text{if } x = q\\
x & \text{otherwise}.
\end{cases}
\end{equation}
Then \(F\) is a combinatorial self similarity by statement \ref{p4.4:s2}. It follows from~\eqref{p4.4:e3} that \(x_1\) and \(x_3\) are fixed points of \(F\). Consequently, \eqref{p4.4:e1} implies \(d(F(x_1), F(q)) = a\) by Lemma~\ref{l2.3}. Using \eqref{p4.4:e3} we can rewrite the last equality as \(d(q, x_1) = a\). Since \(x_1 = p\), the equality \(d(q, x_1) = a\) implies \eqref{p4.4:e2}.

If the sets \(\{p, q\}\) and \(\{x_1, x_2, x_3\}\) are disjoint, then considering the triangles \(\{x_1, x_2, q\}\) and \(\{x_1, p, q\}\) we can get \eqref{p4.4:e2} as above.
\end{proof}

\begin{theorem}\label{t2.4}
Let \((X, d)\) be a nonempty semimetric space. Then the following statements are equivalent:
\begin{enumerate}
\item \label{t2.4:s1} At least one of the following conditions has been fulfilled:
\begin{enumerate}
\item \label{t2.4:s1.1} \((X, d)\) is strongly rigid;
\item \label{t2.4:s1.2} \((X, d)\) is discrete;
\item \label{t2.4:s1.3} \((X, d)\) is weakly rigid and all three-point subspaces of \((X, d)\) are isometric.
\end{enumerate}
\item \label{t2.4:s2} \(\Cs(X, d) = \Sym(X)\) holds.
\end{enumerate}
\end{theorem}

\begin{proof}
\(\ref{t2.4:s1} \Rightarrow \ref{t2.4:s2}\). It suffices to show that the implication \(\ref{t2.4:s1.3} \Rightarrow \ref{t2.4:s2}\) is valid.

Let \ref{t2.4:s1.3} hold. Then, by Proposition~\ref{p2.2}, \((X, d)\) is combinatorially similar to the rectangle from Example~\ref{ex2.1}. Since every permutation of vertices of this rectangle is a combinatorial self similarity, we obtain the equality \(\Cs(X, d) = \Sym(X)\) by Lemma~\ref{l2.6}.

\(\ref{t2.4:s2} \Rightarrow \ref{t2.4:s1}\). Let every permutation of \(X\) be a combinatorial self similarity of \((X, d)\). We must prove the validity of \ref{t2.4:s1}.

Let us consider first the case when \(|X| \leqslant 2\). In this case every semimetric \(d\) on \(X\) is discrete and strongly rigid. Thus, we have the validity of the implication \(\ref{t2.4:s2} \Rightarrow \ref{t2.4:s1.1}\) and \(\ref{t2.4:s2} \Rightarrow \ref{t2.4:s1.2}\) if \(|X| \leqslant 2\).

Let \(|X| = 3\) hold. Then exactly one from the implications \(\ref{t2.4:s2} \Rightarrow \ref{t2.4:s1.1}\), \(\ref{t2.4:s2} \Rightarrow \ref{t2.4:s1.2}\) is true by Lemma~\ref{l4.2}.

Suppose now that \(|X| = 4\) holds, but \(d \colon X^2 \to \RR\) is neither discrete nor strongly rigid. Since \((X, d)\) is not strongly rigid, there are two-point subsets \(\{x, y\}\), \(\{u, v\}\) of \(X\) and \(a \in (0, \infty)\) such that
\begin{equation}\label{t2.4:e6}
\{x, y\} \neq \{u, v\}
\end{equation}
and
\begin{equation}\label{t2.4:e7}
d(x, y) = d(u, v) = a.
\end{equation}
We claim that \(\)
\begin{equation}\label{t2.4:e8}
\{x, y\} \cap \{u, v\} = \varnothing.
\end{equation}

Indeed, if \eqref{t2.4:e8} does not hold, then, using~\eqref{t2.4:e6}, we obtain that \(Y = \{x, y\} \cup \{u, v\}\) is a triangle. By Lemma~\ref{l4.3}, from \ref{t2.4:s2} follows the equality
\begin{equation}\label{t2.4:e8.1}
\Cs(Y, d|_{Y^2}) = \Sym(Y).
\end{equation}
Now using Lemma~\ref{l4.2}, we see that the semimetric \(d|_{Y^2}\) is either strongly rigid or discrete. If \(d|_{Y^2}\) is strongly rigid, then from \(\{x, y\} \neq \{u, v\}\) it follows that \(d(x, y) \neq d(u, v)\), contrary to \eqref{t2.4:e7}. Consequently, \(d|_{Y^2}\) is discrete. By Proposition~\ref{p4.4}, the last statement implies that \(d \colon X^2 \to \RR\) is discrete contrary to our supposition. Equality~\eqref{t2.4:e8} follows.

From~\eqref{t2.4:e8} follows the equality \(X = \{x, y, u, v\}\). Applying Lemma~\ref{l2.3} to the permutation
\[
\begin{pmatrix}
x & y & u& v\\
x & u & y& v
\end{pmatrix},
\]
we obtain the equality
\begin{equation}\label{t2.4:e9}
d(x, u) = d(y, v) = b
\end{equation}
for some \(b \in (0, \infty)\) and, analogously, for the permutation
\[
\Phi := \begin{pmatrix}
x & y & u& v\\
v & x & u& y
\end{pmatrix},
\]
Lemma~\ref{l2.3} implies
\begin{equation}\label{t2.4:e10}
d(v, x) = d(u, y) = c
\end{equation}
for some \(c \in (0, \infty)\). To complete the proof of validity \(\ref{t2.4:s2} \Rightarrow \ref{t2.4:s1.3}\) it suffices to show that the numbers \(a\), \(b\), \(c\) are pairwise distinct, that can be done in the same way as in the proof of equality \eqref{t2.4:e8}.

Let us consider now the case \(|X| \geqslant 5\). We claim that the inequality \(|X| \geqslant 5\) and \eqref{t2.4:e8.1} imply that \(d\) is strongly rigid or discrete. To prove it, suppose contrary that \(d\) is neither strongly rigid nor discrete and construct a four-point set \(Y \subseteq X\) such that \((Y, d|_{Y^{2}})\) is combinatorially similar to the metric space from Example~\ref{ex2.1}. We will construct \(Y\) by modification of our proof of validity \(\ref{t2.4:s2} \Rightarrow \ref{t2.4:s1}\) when \(|X| = 4\).

Let \(\{x, y\}\) and \(\{u, v\}\) be two point subsets of \(X\) satisfying \eqref{t2.4:e6} and \eqref{t2.4:e7} (with some \(a > 0\)). Write \(Y := \{x, y\} \cup \{u, v\}\). As in the case \(|X| = 4\), we obtain \eqref{t2.4:e8} and \eqref{t2.4:e8.1}. From \eqref{t2.4:e6} and \eqref{t2.4:e7} it follows that \(d|_{Y^2}\) is not strongly rigid. If \(d|_{Y^2}\) is discrete, then, by Proposition~\ref{p4.4}, equality \eqref{t2.4:e8.1} implies that \(d\) is also discrete, contrary to our assumption. Moreover, arguing as in the case \(|X| = 4\), we can show that equalities \eqref{t2.4:e7}, \eqref{t2.4:e9} and \eqref{t2.4:e10} are valid with pairwise distinct \(a\), \(b\), \(c \in (0, \infty)\). Thus, we have \(|Y| = 4\) and \eqref{t2.4:e8.1}. It was shown above that \(\ref{t2.4:s2} \Rightarrow \ref{t2.4:s1.3}\) is true if \(|X| = 4\) and \(d\) is neither strongly rigid nor discrete. Consequently, \((Y, d|_{Y^2})\) is combinatorially similar to the metric space from Example~\ref{ex2.1}.

\begin{figure}[htb]
\begin{tikzpicture}
\def\xx{1.5cm}
\def\dx{0.6cm}
\def\dy{1cm}
\coordinate [label=left: \(x\)] (x1) at (0, 0);
\coordinate [label=left: \(y\)] (x2) at (\dx, \dy);
\coordinate [label=right: \(u\)] (x3) at (\xx + \dx, \dy);
\coordinate [label=right: \(v\)] (x4) at (\xx, 0);
\coordinate [label=right: \(p\)] (x5) at (.5*\xx, 3*\dy);
\draw [fill] (x1) circle (2pt);
\draw [fill] (x2) circle (2pt);
\draw [fill] (x3) circle (2pt);
\draw [fill] (x4) circle (2pt);
\draw [fill] (x5) circle (2pt);

\node at (4*\dx, 3*\dy) {\(P\)};

\draw (x1) -- (x2) -- (x3) -- (x4) -- (x1) -- (x5) -- (x2);
\draw (x3) -- (x5) -- (x4);
\end{tikzpicture}
\caption{From the quadruple \(\{x, y, u, v\}\) to the pyramid \(P\).}
\label{fig2}
\end{figure}
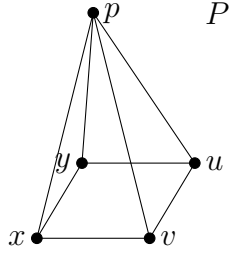

Since \(|X| \geqslant 5\) holds, there is a point \(p \in X \setminus Y\). Write
\[
P = \{x, y, u, v, p\}
\]
(see Figure~\ref{fig2}). Then we have the equality
\begin{equation}\label{t2.4:e11}
\Cs(P, d|_{P^2}) = \Sym(P)
\end{equation}
by Lemma~\ref{l4.3}. Furthermore, for every four-point set \(S \subseteq P\), equality \eqref{t2.4:e11} implies
\begin{equation}\label{t2.4:e12}
\Cs(S, d|_{S^2}) = \Sym(S).
\end{equation}
Now applying \eqref{t2.4:e12} and using the permutations
\begin{align*}
 & \begin{pmatrix}
x & y & u & v & p\\
p & y & u & v & x
\end{pmatrix}, &
 & \begin{pmatrix}
x & y & u & v & p\\
x & p & u & v & y
\end{pmatrix},\\
&  \begin{pmatrix}
x & y & u & v & p\\
x & y & p & v & u
\end{pmatrix} \textrm{\ and} &
 &  \begin{pmatrix}
x & y & u & v & p\\
x & y & u & p & v
\end{pmatrix}
\end{align*}
we see that every four-point subspace \((S, d|_{S^2})\) of \((P, d|_{P^2})\) is combinatorially similar to \((Y, d|_{Y^2})\). In addition, it is easy to see that, for every four-point \(S \subseteq P\), there is a three-point set \(T\) such that
\begin{equation}\label{t2.4:e13}
T \subseteq S \cap Y.
\end{equation}
Since for every three-point set \(T\) we have the equality
\[
d(T^2) = d(Y^2) = \{0, a, b, c\},
\]
\eqref{t2.4:e12} implies the equality $d(S^2) = \{0, a, b, c\}$ for every four-point \(S \subseteq P\), because \((S, d|_{S^2})\) is combinatorially similar to \((Y, d|_{Y^2})\).

Let us consider the ribs \(\<p, x>\), \(\<p, y>\), \(\<p, u>\), \(\<p, v>\) of the pyramid \(P\). Then, by Pigeonhole principle, we can find two different points in \(Y\), say \(u\) and \(v\), such that
\begin{equation}\label{t2.4:e15}
d(p, u) = d(p, v).
\end{equation}
Write \(Z := \{p, u, v\}\). Then we have \(\Cs(Z, d|_{Z^2}) = \Sym(Z)\) by Lemma~\ref{l4.3}. Consequently, \eqref{t2.4:e15} implies that \((Z, d|_{Z^2})\) is discrete by Lemma~\ref{l4.2}. Thus, \((X, d)\) is also discrete by Proposition~\ref{p4.4}. The discreteness of \((X, d)\) contradicts our assumption that \((X, d)\) is neither strongly rigid nor discrete. The proof is completed.
\end{proof}

\begin{corollary}\label{c4.6}
The following conditions are equivalent for every nonempty set \(X\):
\begin{enumerate}
\item\label{c4.6:s1} \(|X| = 4\).
\item\label{c4.6:s2} There is a semimetric \(d \colon X^2 \to \RR\) such that \(\Cs(X, d) = \Sym(X)\) but \(d\) is neither strongly rigid nor discrete.
\end{enumerate}
\end{corollary}

\begin{corollary}\label{c2.10}
Let \((X, d)\) be a semimetric space such that the inequality
\begin{equation}\label{c2.10:e1}
|X| > \mathfrak{c}
\end{equation}
holds, where \(\mathfrak{c}\) is the cardinality of the continuum. Then
\begin{equation}\label{c2.10:e2}
\Cs(X, d) = \Sym(X)
\end{equation}
holds if and only if we have
\begin{equation}\label{c2.10:e3}
\Iso(X, d) = \Sym(X).
\end{equation}
\end{corollary}

\begin{proof}
It suffices to show that \eqref{c2.10:e2} implies \eqref{c2.10:e3}.

Let \eqref{c2.10:e2} hold. Then, by Corollary~\ref{c4.6}, \(d\) is strongly rigid or discrete. Suppose \(d\) is strongly rigid and let \(p\) be a point of \(X\). Then the mapping
\[
X \ni x \mapsto d(x, p) \in \RR
\]
is injective. Hence, we have the inequality \(|X| \leqslant \mathfrak{c}\) contrary to \eqref{c2.10:e1}. Thus, \(d\) is discrete, that implies \eqref{c2.10:e3}.
\end{proof}

We conclude the paper by an interesting example of four-point metric space \((Y, \rho)\) satisfying the equality \(\Cs(Y, \rho) = \Sym(Y)\).

\begin{example}
Recall (see \cite{Blumenthal1953} for instance) that a four-point metric space \((X, \rho)\) is called a \emph{pseudolinear quadruple} if for a suitable enumeration of the points we have
\begin{align*}
\rho(x_1, x_2) & = \rho(x_3, x_4) = s, & \rho(x_2, x_3) &= \rho(x_4, x_1) = t, \\
\rho(x_2, x_4) & = \rho(x_3, x_1) = s + t,
\end{align*}
with some positive reals \(s\) and \(t\). The pseudolinear quadruple \((X, \rho)\) is combinatorially similar to the rectangle from Example~\ref{ex2.1} iff \(s \neq t\).

The pseudolinear quadruples and their higher-dimensional modifications appeared for the first time in the famous paper of Menger \cite{Men1928MA} who in particular gave a criterion for the embeddability of metric spaces into \(\RR^n\). According to Menger, the pseudolinear quadruples are characterized as the metric spaces not isometric to any subset of $\mathbb{R}$ whose every triple of points embeds isometrically into \(\RR\). There is an elementary proof of this fact in \cite{DD2009UJ}.
\end{example}

\section*{Funding}

Oleksiy Dovgoshey was supported by Volkswagen Stiftung Project ``From
Modeling and Analysis to Approximation''.

%\bibliographystyle{amsplain}
%\bibliography{bib2021.04}

\begin{thebibliography}{10}

\bibitem{BP2017JNCA}
M.~Bessenyei and Z.~P\'{a}les, \emph{A contraction principle in semimetric
  spaces}, J. Nonlinear Convex Anal. \textbf{18} (2017), no.~3, 515--524.

\bibitem{BDKP2017AASFM}
V.~Bilet, O.~Dovgoshey, M.~K\"{u}\c{c}\"{u}kaslan, and E.~Petrov,
  \emph{{Minimal universal metric spaces}}, Ann. Acad. Sci. Fenn. Math.
  \textbf{42} (2017), no.~2, 1019--1064.

\bibitem{BDS2021pNUAA}
V.~Bilet, O.~Dovgoshey, and R.~Shanin, \emph{Ultrametric preserving functions
  and weak similarities of ultrametric spaces}, p-Adic Numbers Ultrametric
  Anal. Appl. \textbf{13} (2021), no.~3, 186--203.

\bibitem{Blumenthal1953}
L.~M. Blumenthal, \emph{{Theory and Applications of Distance Geometry}},
  Clarendon Press, Oxford, 1953.

\bibitem{DF1998DM}
J.~Diatta and B.~Fichet, \emph{Quasi-ultrametrics and their 2-ball
  hypergraphs}, Discrete Math. \textbf{192} (1998), no.~1--3, 87--102.

\bibitem{Dov2019IEJA}
O.~Dovgoshey, \emph{{Semigroups generated by partitions}}, Int. Electron. J.
  Algebra \textbf{26} (2019), 145--190.

\bibitem{DovBBMSSS2020}
\bysame, \emph{Combinatorial properties of ultrametrics and generalized
  ultrametrics}, Bull. Belg. Math. Soc. Simon Stevin \textbf{27} (2020), no.~3,
  379--417.

\bibitem{DD2009UJ}
O.~Dovgoshey and D.~Dordovskyi, \emph{Betweenness relation and isometric
  imbeddings of metric spaces}, Ukrainian Math. J. \textbf{61} (2009), no.~10,
  1556--1567.

\bibitem{DLAMH2020}
O.~Dovgoshey and J.~Luukkainen, \emph{Combinatorial characterization of
  pseudometrics}, Acta Math. Hungar \textbf{161} (2020), no.~1, 257--291.

\bibitem{DP2013AMH}
O.~Dovgoshey and E.~Petrov, \emph{{Weak similarities of metric and semimetric
  spaces}}, Acta Math. Hungar \textbf{141} (2013), no.~4, 301--319.

\bibitem{DS2021aa}
O.~Dovgoshey and R.~Shanin, \emph{Uniqueness of best proximity pairs and
  rigidity of semimetric spaces}, arXiv:2201.04380v2 (2022), 1--32.

\bibitem{DH2017JFPTA}
N.~V. Dung and V.~T.~L. Hang, \emph{On regular semimetric spaces having strong
  triangle functions}, J. Fixed Point Theory Appl. \textbf{19} (2017), no.~3,
  2069--2079.

\bibitem{Fre1906RdCMdP}
M.~Fr\'{e}chet, \emph{Sur quelques points de calcul fonctionnel}, Rend. del
  Circ. Mat. di Palermo \textbf{22} (1906), 1--74.

\bibitem{Gre2016CMP}
D.~J. Greenhoe, \emph{Properties of distance spaces with power triangle
  inequalities}, Carpathian Math. Publ. \textbf{8} (2016), no.~1, 51--82.

\bibitem{JT2020R}
J.~Jachymski and F.~Turobo\'{s}, \emph{On functions preserving regular
  semimetrics and quasimetrics satisfying the relaxed polygonal inequality},
  RACSAM \textbf{114} (2020), no.~3, 1--11.

\bibitem{Janos1972}
L.~Janos, \emph{A metric characterization of zero-dimensional spaces}, {Proc.
  Amer. Math. Soc.} \textbf{31} (1972), no.~1, 268--270.

\bibitem{KS2015JFPTA}
W.~A. Kirk and N.~Shahzad, \emph{{Fixed points and Cauchy sequences in
  semimetric spaces}}, J. Fixed Point Theory Appl. \textbf{17} (2015), no.~3,
  541--555.

\bibitem{KV1984}
E.~Kunen and J.~E. Vaughan, \emph{Handbook of set-theoretic topology},
  North-Holland Publishing Co., Amsterdam, 1984.

\bibitem{Martin1977}
H.~W. Martin, \emph{Strongly rigid metrics and zero dimensionality}, Proc.
  Amer. Math. Soc. \textbf{67} (1977), no.~1, 157--161.

\bibitem{Men1928MA}
K.~Menger, \emph{{Untersuchunger \"{u}ber allgemeine Metrik. I?III}}, Math.
  Ann. \textbf{100} (1928), 75--163.

\bibitem{Pet2018pNUAA}
E.~Petrov, \emph{{Weak similarities of finite ultrametric and semimetric
  spaces}}, p-Adic Numbers Ultrametr. Anal. Appl. \textbf{10} (2018), no.~2,
  108--117.

\bibitem{PS2022AFM}
E.~Petrov and R.~Salimov, \emph{On quasisymmetric mappings in semimetric
  spaces}, Annales Fennici Mathematici \textbf{47} (2022), 723--745.

\bibitem{Wilson1931}
W.~A. Wilson, \emph{{On semi-metric spaces}}, Am. J. Math. \textbf{53} (1931),
  361--373.

\end{thebibliography}

%\providecommand{\bysame}{\leavevmode\hbox to3em{\hrulefill}\thinspace}
%\providecommand{\MR}{\relax\ifhmode\unskip\space\fi MR }
%\providecommand{\MRhref}[2]{%
%  \href{http://www.ams.org/mathscinet-getitem?mr=#1}{#2}
%}
%\providecommand{\href}[2]{#2}

\end{document}